\documentclass[review,12pt,number]{elsarticle}
\usepackage{amsmath,amsthm,amssymb}
\usepackage{graphicx}
\usepackage{indentfirst}
\usepackage{eps fig}

\newtheorem{thm}{Theorem}[section]
\newtheorem{cor}{Corollary}[section]
\newtheorem{lemm}{Lemma}[section]
\newtheorem{example}{Example}
\newtheorem{que}{Question}

\newtheorem{definition}{Definition}[section]
\newtheorem{rem}{Remark}
\numberwithin{equation}{section}
\journal{\empty}
\begin{document}

\begin{frontmatter}

\title{\textsc{Baker omitted value}}

\author{Tarun Kumar Chakra}
\author{Gorachand Chakraborty}
\author[]{Tarakanta Nayak\corref{cor1}\fnref{1}}
\ead{tnayak@iitbbs.ac.in}
\address{School of Basic Sciences, Indian Institute of Technology Bhubaneswar, India}

\cortext[cor1]{Corresponding author}
\fntext[1]{The corresponding author is supported by Department of Science and Technology, Govt of India through a Fast Track Project (SR-FTP-MS019-2011).}
\begin{abstract}
	 We define  Baker omitted value, in short \textit{$\textit{bov}$}, of an entire or meromorphic function $f$ in the complex plane as an
	 omitted value for which there exists $r_0>0$ such that for each ball $D_r (a)$ centered at $a$ and with radius $r$ satisfying $0<r<r_{0}$, every component of the boundary of $f^{-1}(D_r (a))$ is bounded. The \textit{$\textit{bov}$} of a function is the only  asymptotic value.  An entire function has   \textit{$\textit{bov}$} if and only if the image of every unbounded curve is unbounded. It follows that an entire function has \textit{$\textit{bov}$} whenever it has a Baker wandering domain. Functions with \textit{$\textit{bov}$} has at most one completely invariant Fatou component. Some examples and counter examples are discussed and questions are proposed for further investigation.
\end{abstract}
\date{\empty}
\end{frontmatter}
\section{Introduction}
Let $f:\mathbb{C}\rightarrow \widehat{\mathbb{C}}=\mathbb{C} \bigcup \{\infty\}$ be a transcendental  entire or meromorphic function with a single essential singularity which we choose to be at $\infty$. A point $a\in \widehat{\mathbb{C}}$ is called a singular value of $f$ if for every open neighborhood $U$ of $a$, there exists a component $V$ of $f^{-1}(U)$ such that $f:V\rightarrow U$ is not injective.  The set of singular values is the closure of all critical values and asymptotic values of $f$. A critical value is the image of a critical point, that is, $f(z_{0})\ \mbox{where}\ f'(z_{0})=0$. A point $a\in \widehat{\mathbb{C}}$ is an asymptotic value of $f$ if there exists a curve $\gamma:[0,\infty)\rightarrow \mathbb{C}$ with $\lim_{t\rightarrow \infty} \gamma(t)=\infty$ such that $ \lim_{t\rightarrow \infty} f(\gamma(t))=a$. An equivalent definition of singular values which is more relevant for this article follows~\cite{berg-ermk}. For $a\in \widehat{\mathbb{C}}$ and $r>0$, let $D_{r}(a)$ be a ball 
with respect to the spherical metric  and choose a component $U_{r}$ of
$f^{-1}(D_{r}(a))$ in such a way that $U_{r_{1}}\subset U_{r_{2}}$ for $0<r_{1}<r_{2}.$ There are two possibilities.
\begin{enumerate}
	\item $\bigcap_{r>0}U_{r}= \{z\}\ \mbox{for}\ z\in\mathbb{C}$: Then $f(z)=a.$ The point $z$ is called an ordinary point if (i) $
	f'(z)\neq 0\ \mbox{and}\ a\in\mathbb{C}$, or (ii) $ z\ \mbox{is\ a\ simple\ pole}.$ The point $z$ is called a critical point if
	$f'(z)=0$ and $a\in \mathbb{C}$, or $z$ is a multiple pole. In
	this case, $a$ is called a critical value and we say that a critical
	point or algebraic singularity lies over $a$.
	
	\item $\bigcap_{r>0}U_{r}= \emptyset$: The choice $r \mapsto U_{r}$ defines a transcendental singularity of $f^{-1}$. We say a
	transcendental singularity lies over $a$. There is a transcendental singularity over $a$ if and only if $a$ is an asymptotic value. The pre-image of every ball centered at an asymptotic value has at least an unbounded component. The singularity lying
	over $a$ is called direct if there exists $r>0$ such that
	$f(z)\neq a$ for all $z\in U_{r}$. The singularity lying over $a$
	is called logarithmic if $f: U(r)\rightarrow D_{r}(a)\setminus
	\{a\}$ is a universal covering for some $r>0$. A singularity is indirect if it is not direct.
\end{enumerate}

A value $z_{0}\in \widehat{\mathbb{C}}$ is said to be an \textit{omitted value} for the function $f$ if $f(z)\neq z_{0}$ for any $z\in\mathbb{C}$. It is easy to note that each singularity lying over an omitted value is direct. In other words, each component of $f^{-1}(B)$ is unbounded for every ball $B$ centered at an omitted value and with sufficiently small radius. However, more than one singularity can lie over an omitted value. For example, there are infinitely many singularities of the inverse of $\tan e^z$ lying over $i$. This article deals with certain type of omitted values over which there is only one singularity and that is not logarithmic.
\par
Let $D_r(a)$ be the ball around $a$ with radius $r$ with respect to the spherical metric.
\begin{definition}(\textbf{Baker omitted value })
	 An omitted value $a \in \widehat{\mathbb{C}}$ of an entire or meromorphic function $f$ is said to be Baker omitted value, in short \textit{$\textit{bov}$}, if there exists $r_0>0$ such that for all $r$ satisfying $0<r<r_0$, each component of the boundary of $f^{-1}(D_r(a))$ is bounded.
	\end{definition}

It follows that $f^{-1}(D_r(a))$ is infinitely connected and each component of $ \mathbb{C}\setminus f^{-1}(D_r(a))$ is bounded (Lemma~\ref{bov-implication1}).  The \textit{$\textit{bov}$} is the only  asymptotic value of the function (Theorem~\ref{oneav}). Consequently, the \textit{$\textit{bov}$} of an entire function must be $\infty$ and the \textit{$\textit{bov}$} of a meromorphic function is always finite. If $g$ is an entire function with a $\textit{bov}$ then $a$ is the $\textit{bov}$ of the meromorphic function $\frac{1}{g}+a, a \in \mathbb{C}$.
Conversely, if $f$ is a meromorphic function with $\textit{bov}$ $a$ then $\infty$ is the $\textit{bov}$ of the entire function $\frac{1}{f-a}$.
\par
The sequence of iterates $\{f^{n}\}_{n>0}$ of $f$ on a domain $\Omega\subseteq \mathbb{C}$ is said to be \textit{normal} if every sequence of functions in $\{f^{n}\}_{n>0}$ contains a subsequence which converges either to a finite
limit function or to $\infty$ uniformly on each compact subsets of $\Omega$. The set of points $z\in \mathbb{C}$ in a neighborhood of which the sequence of iterates $\{f^{n} \}_{n>0}$ is defined and forms a normal family is called the \textit{Fatou set} of $f$ and is denoted by $\mathcal{F}(f)$. The \textit{Julia set}, denoted by $\mathcal{J}(f)$, is the complement of $\mathcal{F}(f)$ in $\widehat{\mathbb{C}}$. The Fatou set is open and the Julia set is perfect. A maximally connected subset of the Fatou set is called a \textit{Fatou component}. For a Fatou component $U$, $U_{k}$ denotes the Fatou component containing $f^{k}(U)$. A Fatou component $U$ is called \textit{wandering} if $U_{n}\neq U_{m}$ for all $n \neq m$. We say a multiply connected Fatou component $U$ surrounds a point $a \in \mathbb{C}$ if there exists a bounded component of $\mathbb{\widehat{C}}-U$ containing $a$.
\begin{definition}(\textbf{Baker wandering domain})  A Baker wandering domain is a wandering component $U$ of $\mathcal{F}(f)$ such that, for $n$ large enough, $U_{n}$ is  bounded, multiply connected and surrounds $0$, and $f^{n}(z)\rightarrow \infty\ \mbox{as}\ n\rightarrow\
	\infty\ \mbox{for}\ z \in U.$
\end{definition}
Note that if $f$ has a Baker wandering domain then all the Fatou components including the Baker wandering domains are bounded.
\par
Let $E$ denotes the class of all transcendental entire functions
and $M$ denotes the class of transcendental meromorphic functions with at least two poles or one pole that is not an omitted value. An omitted value  of a function  belonging to the class ${M}$  determines a number of important aspects of the dynamics of the function~\cite{tk-zheng}.
In this article, we first state and prove the existence of Baker omitted value and  investigate some of its dynamical implications. Theorem~\ref{boviff} proves that an entire function has the $\textit{bov}$ if and only if the image of every unbounded curve is unbounded. It is proved in Theorem~\ref{bwdbov} that if an entire function has a Baker wandering domain then it has $\textit{bov}$.  A sufficient condition is also found for meromorphic functions to have $\textit{bov}$.
\par
A Fatou component $U$ of a function $f $ is called completely invariant if $f^{-1}(U) \subseteq U$ and $f(U) \subseteq U$. It is proved in Theorem~\ref{onecifc} that every meromorphic function with  $\textit{bov}$ has at most one completely invariant Fatou component. This is a stronger version of an unresolved conjecture which
states that the number of completely invariant Fatou components of a meromorphic function is at most two.
\par
A number of properties of a function with $\textit{bov}$ are found in Section 2. Section 3 discusses the existence of $\textit{bov}$ and relevant examples. We investigate completely invariant Fatou components of functions with $\textit{bov}$ in Section 4. Some questions are proposed for further investigation in Section 5.
 A comparison of properties of entire and meromorphic functions with $\textit{bov}$ is given in Table~\ref{table:nonlin}.
\par
Each function considered in this article is either entire or meromorphic. Every ball is with respect to the spherical metric. We denote the ball centered at $a$ and with radius $r$ by $D_r (a)$  throughout the article. Let $\mathcal{J}_z, z \in \widehat{\mathbb{C}}$ denote the component of the Julia set containing $z$.
\section{Implications}
First, we analyze the pre-image of small balls centered at $\textit{bov}$.
\begin{lemm}
	If $a$ is the $\textit{bov}$ of $f$ then there is $r_0$ such that for all $0< r< r_0$,
	$f^{-1}(D_r (a))$ is infinitely connected and each  component of $ \mathbb{C}\setminus f^{-1}(D_r(a))$ is bounded.
	\label{bov-implication1}
\end{lemm}
\begin{proof}
By the definition of $\textit{bov}$, there is $r_0>0$ such that for all $r< r_0$, each component of the boundary of $f^{-1}(D_r (a))$ is bounded. Let $r < r_0$. Since $a$ is omitted, each component of $f^{-1}(D_r(a))$ is unbounded.
	If $f^{-1}(D_r(a))$ is disconnected then it has at least two components and each of them is unbounded.
Further, the boundary of each of them has an unbounded component.  However, this is not true proving that $f^{-1}(D_r(a))$ is connected for $r< r_0$.
	\par
	If  $f^{-1}(D_r(a))$ is simply connected or finitely connected then it contains a neighborhood of $\infty$ punctured at $\infty$. By the Casorati-Weirstrass Theorem, $f( f^{-1}(D_r(a)))$ contains $\mathbb{C}$ except possibly two points. This is a contradiction as $f( f^{-1}(D_r(a)))=D_r(a)$. Therefore $f^{-1}(D_r(a))$ is infinitely connected. It follows that each component  of $ \mathbb{C}\setminus f^{-1}(D_r(a))$  is bounded.
\end{proof}
Singular values of a function are restricted whenever
it  has $\textit{bov}$. For proving this, we prove a lemma which may be also of independent importance.

\begin{lemm}
	If $a$ is the $\textit{bov}$ of $f $ and $B$ is any ball not containing $a$ in its closure then
	each component of $f^{-1}(B)$ is bounded.
	\label{bovnotinball}
\end{lemm}
\begin{proof}
	If $D$ is a ball centered at $a$ and $D \bigcap B =\emptyset$ then $f^{-1}(B)$ is contained
	in $\mathbb{C} \setminus f^{-1}(D)$. Further, each component of $f^{-1}(B)$ is contained in a component of $\mathbb{C} \setminus f^{-1}(D)$. But each component of $\mathbb{C} \setminus f^{-1}(D)$ is bounded.
	Therefore, each component of $f^{-1}(B)$ is bounded.
\end{proof}
\begin{thm}
	If $a$ is the $\textit{bov}$ of $f$ then $a$ is the only asymptotic value of $f$.
	\label{oneav}
\end{thm}
\begin{proof}
	Let $f$ have an asymptotic value $b$ in $\widehat{\mathbb{C}}$ and $b \neq a$.
	Considering a ball $B$ centered at $b$ whose closure is disjoint from $a$, it follows from the Lemma~\ref{bovnotinball} that
	each component of $f^{-1}(B)$ is bounded. But this is not true as $b$ is an asymptotic value.
\end{proof}
The above result says that if $f$ has $\textit{bov}$ then every other singular value of $f$ is a critical value. Some other important remarks follow.

\begin{rem}
	\begin{enumerate}
		\item Each $\textit{bov}$ is an asymptotic value. In view of Theorem~\ref{oneav}, it is clear that there is at most one $\textit{bov}$ for a function.
		\item The point at $\infty$ is always an asymptotic value for an entire function. Therefore, if an entire function has $\textit{bov}$ then it is $\infty$.
		\item If a meromorphic function has $\textit{bov}$ then it is a finite complex number. This is because $\infty$ is not an omitted value for a meromorphic function.
		\item If a point is finitely taken then it is an asymptotic value. Therefore, if  $a\in \widehat{\mathbb{C}}$ is the $\textit{bov}$ of $f$ then for each $z\in \widehat{\mathbb{C}}-\{a\} $ there are infinitely many points in $\mathbb{C}$ that are mapped to $z$ by $f$.
		In particular, each entire function with $\textit{bov}$ takes every finite point infinitely often and
		each meromorphic function with $\textit{bov}$ has infinitely many poles.
	\end{enumerate}
\label{rem}\end{rem}

\begin{figure}[ht!]
	\centering
	\includegraphics[scale=0.4]{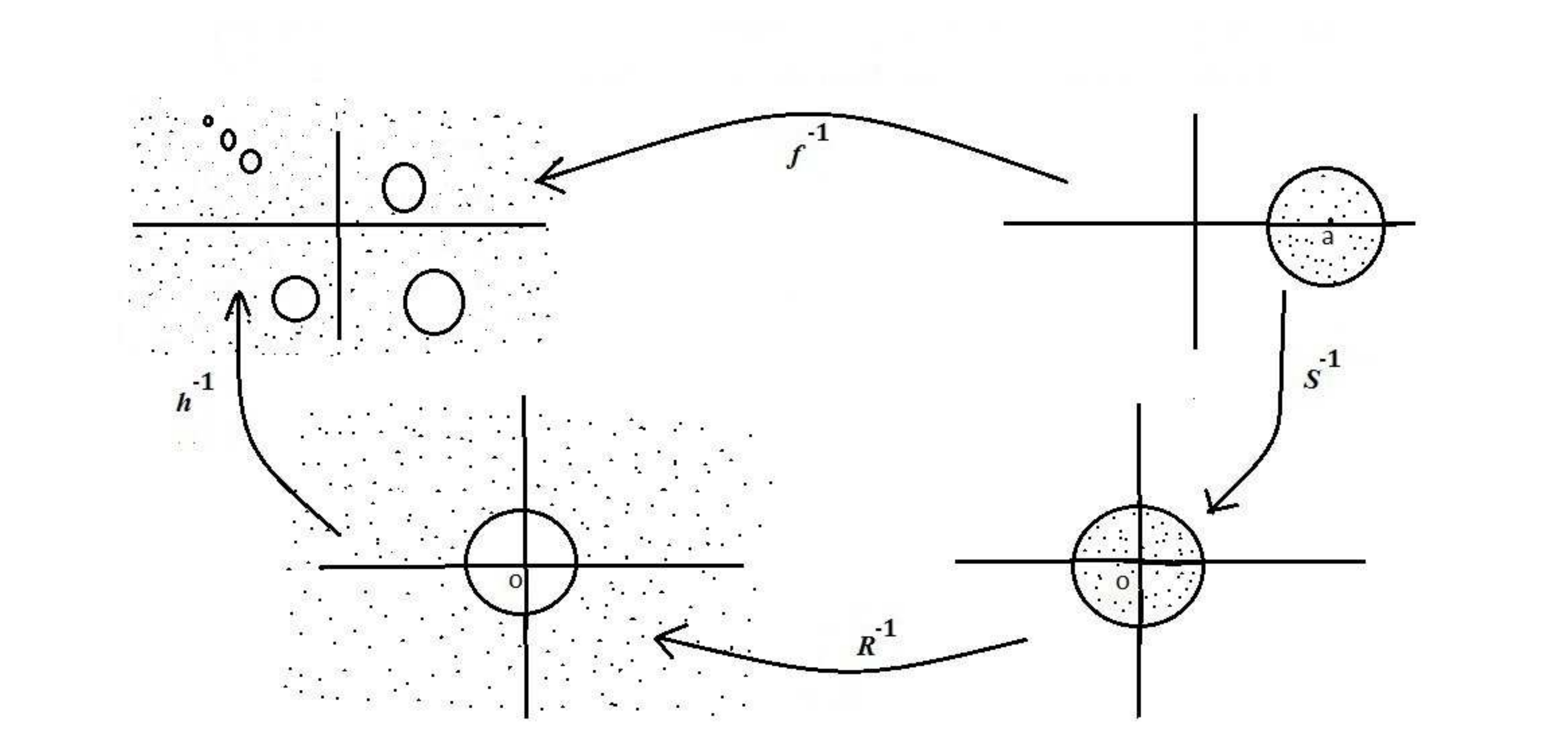}
	\caption{Pre-image of a ball centered at a Baker omitted value of $f$, where $f=\frac{1}{h}+a=S\circ R\circ h; S(z)=z+a, R(z)=1/z$. \label{overflow}}
\end{figure}
It follows from Theorem~\ref{oneav} that an entire function $f$ has $\textit{bov}$ then it has no finite asymptotic value. However,
the converse is not true in general as given in the following example.
\begin{example}
	It is well known that $f(z)= \sin z$  has no finite asymptotic value. Note that each component of the pre-image of every simply connected and bounded domain under an entire function is simply connected (See the proof of Theorem~\ref{boviff}). The pre-image of every simply connected domain $B$ containing $[-1,1]$ is a simply connected domain containing the whole real line.  This is because $\sin z \in [-1,1]$ if and only if $z$ is a real number. If $D$ is any neighborhood of $\infty$ not containing $[-1,1]$ then $\mathbb{C} \setminus D$ is a simply connected domain containing $[-1,1]$. Thus $f^{-1}(\mathbb{C} \setminus D)$ is a simply connected domain containing the real line and $f^{-1}( D)$ is not connected.  Therefore  $\infty$ is not a $\textit{bov}$ of $f$. However, if $D$   contains $-1$ and $1$ then $f^{-1}(\widehat{\mathbb{C}}\setminus D)$ is
	a disjoint union of infinitely many simply connected bounded domains. Thus $f^{-1}( D)$ is infinitely connected and each component of its complement is bounded.
	
\end{example}
Though singularity of $f^{-1}$ at a point signifies the local behavior of a branch of $f^{-1}$ at the point,  it has global implications in case of $\textit{bov}$. Below we state and prove a special property of $\textit{bov}$, special in the sense that what is true for smaller balls around a $\textit{bov}$ is also true for bigger balls.
\begin{lemm}
	Let   $a\in \widehat{\mathbb{C}}$ be the $\textit{bov}$ of $f$. Then for all $r>0$, $f^{-1}(D_r(a))$ is infinitely connected and each component of $ \mathbb{C}\setminus f^{-1}(D_r(a))$  is bounded.
\label{bov-implication2}\end{lemm}
\begin{proof}
	If $f$ is entire with $\textit{bov}$ then $\frac{1}{f} + a, a \in \mathbb{C}$ is meromorphic having $a$ as its $\textit{bov}$. Since $\frac{1}{z} + a$ is a bijection,  we assume without loss of generality that $f$ is a transcendental meromorphic function and $a \in \mathbb{C}$ is the $\textit{bov}$ of $f$ (See Figure 1). Then there exists $r_0>0$ such that for all $r$ satisfying $0<r<r_0$, $f^{-1}(D_r(a))$ is infinitely connected and each component  of $\mathbb{C} \setminus f^{-1}(D_r(a))$ is bounded by Lemma~\ref{bov-implication1}.
\par
For $r>r_0> s$, the set  $f^{-1}(D_r(a))$ contains  $f^{-1}(D_s(a))$. Since $f^{-1}(D_s(a))$ is connected, there is a component of  $f^{-1}(D_r(a))$ containing  $f^{-1}(D_s(a))$. If there is another component of  $f^{-1}(D_r(a))$ then it must be contained in a component of $\mathbb{C} \setminus f^{-1}(D_s(a))$ and hence bounded. But each component of the pre-image of every ball centered at an omitted value is unbounded. Therefore it is proved that $f^{-1}(D_r(a))$ has a single component. In other words, $f^{-1}(D_r(a))$ is connected.
\par
Let $D$ be a component of $\mathbb{C} \setminus f^{-1}(D_r(a))$. Then it is easy to note that $D$ is contained in a component of $\mathbb{C} \setminus f^{-1}(D_s(a))$. Therefore $D$ is bounded. Since each such $D$ contains a pole and $f$ has infinitely many poles by Remark~\ref{rem}(4), $f^{-1}(D_r(a))$ is infinitely connected.
	This completes the proof.
\end{proof}
\section{Existence}
A necessary and sufficient condition for existence of $\textit{bov}$ for entire functions is the content of the next theorem.
\begin{thm}
	Let $f \in E$. Then for each unbounded curve $\gamma$, $f(\gamma)$ is unbounded if and only if  $\infty$ is the $\textit{bov}$ of $f$.
	\label{boviff}
\end{thm}
\begin{proof}
	Let $\infty$ be the $\textit{bov}$ of $f$ and $\gamma$ be an unbounded curve.
	Then it follows from Lemma~\ref{bov-implication2} that for each ball $D_r$
	centered  at $\infty$ with radius $r$,  $f^{-1}(D_r)$ is infinitely connected and each component of $\mathbb{C} \setminus f^{-1}(D_r)$ is bounded.  In other words, $f^{-1}(D_r) \cap \gamma \neq \emptyset$ for every $r>0$. That means $f(\gamma)$ intersects $D_r$ for every $r$ and hence $f(\gamma)$ is unbounded.
	\par
	To prove the converse, let $f(\gamma)$ be unbounded for each unbounded curve $\gamma$.
Let $B$ be a simply connected and bounded domain and $D$ be a component of $f^{-1}(B)$.
	If $D$ is unbounded then there exists an unbounded curve $\gamma$ in $D$ and consequently  $f(\gamma)$ is unbounded. But it cannot be true since $f(z) \in B$ for all $z \in D$. Thus each $D$ is bounded.
	Now, suppose that there is a bounded component of $\mathbb{C}\setminus D$.
	Let it be $C$ and consider $w\in \partial f(C)$. If $f(z) =w$ then $z$ does not belong to the interior
	of $C$ by the Open Mapping Theorem. Consider a sequence $\{w_n\}_{n>0}$ in $f(C)$ such that
	$w_n \rightarrow w$ as $n \to \infty$. Then there is  $z_n \in C$ such that $f(z_n)=w_n$. Each limit point of the sequence $\{z_n\}_{n>0}$ is in $\overline{C}$. If $z$ is such a limit point then by the continuity of $f$ at $z$,
 $f(z)=w$ and $z\in \partial C$. Now taking
	a sequence $\{b_n\}_{n>0}$  in $D$ such that $b_n \rightarrow z$ as $n \to \infty$ it follows, by continuity of  $f$ at $z$,
	that $f(b_n) \to f(z)$ as $n \to \infty$. But $f(b_n) \in f(D)=B$ gives that $ f(z) \in \partial B$.
	Thus $f(z)=w\in \partial B$ proving that $\partial f(C)\subseteq \partial B$.
	If $\partial f(C)\subsetneq \partial B $ then $f(C)$ is unbounded. Also,
	if $\partial f(C)= \partial B$ then $f(C)=\mathbb{C} \setminus B$ and hence $f(C)$ is unbounded.
	However, this is not possible as $f$ is entire and $C$ is bounded.
	Thus $\mathbb{C}\setminus D$ has no bounded component and hence each $D$ is simply connected.
	Now, it follows from the Casorati-Weirstrass Theorem that $f^{-1}(B)$  has infinitely many components.
	\par
	Let $N$ be a neighborhood of $\infty$. Then $\mathbb{C}\setminus N$ is a bounded simply connected domain.
	There are infinitely many components of $f^{-1}(\mathbb{C}\setminus N) $ and each such component is bounded by the conclusion of the previous paragraph. Let these components be $D_i, ~i=1,2,3,...$.
	We claim that $f^{-1}(N)= \mathbb{C}\setminus \bigcup_{i=0}^{\infty} D_i$.
	If $z \in f^{-1}(N)$ then $f(z) \in N$ and consequently $z\in \mathbb{C}\setminus \bigcup_{i=0}^{\infty} D_i$. Conversely letting $z\in \mathbb{C}\setminus \bigcup_{i=0}^{\infty} D_i$ it follows that $f(z)\in N$. Consequently, $z\in f^{-1}(N)$. Therefore  $f^{-1}(N)= \mathbb{C}\setminus \bigcup_{i=0}^{\infty} D_i$. This implies that $\infty$ is the $\textit{bov}$ of $f$.\\
	
\end{proof}
Baker omitted value is a special property of the inverse of $f$ whereas
Baker wandering domain is a dynamical aspect of a function i.e., a property of $f^n$.
The following result connecting $\textit{bov}$ with Baker wandering domains follows as a consequence of Theorem~\ref{boviff}.
\begin{thm}
	If an entire function has a Baker wandering domain then it has the $\textit{bov}$.
	\label{bwdbov}
\end{thm}
\begin{proof}
	Let $W$ be a Baker wandering domain of an entire function $f$.
	Then every unbounded curve $\gamma$ intersects $W_n$ and $f(\gamma)$ intersects $W_{n+1}$ for infinitely many values of $n$. Since $f^n(z) \to \infty$ for all $z \in W$ by the definition of Baker wandering domain, the intersection $ \bigcup_{n>0} f(\gamma) \bigcap W_{n+1}$ is unbounded.  Therefore $f(\gamma)$ is unbounded.  Now it follows from Theorem~\ref{boviff} that $\infty$ is the $\textit{bov}$ of $f$.
\end{proof}
\begin{example}
	This is the classical example given by Baker~\cite{baker}. The entire function $g(z)$ given by the canonical product
	$$g(z)=Cz^2\prod^{\infty}_{n=1}\left( 1+\frac{z}{r_n} \right),\    1<r_{1}<r_{2}<...,\ C>0, $$ where $r_n$ satisfies some growth condition is an example of entire function having a Baker wandering domain. Therefore, $g$ has  $\textit{bov}$ and that is, of course, $\infty$.\end{example}
\begin{rem}
	If $f \in M$ has a Baker wandering domain then it has no finite
	asymptotic value by Theorem 9 ~\cite{tk-zheng}.
	In particular, it has no $\textit{bov}$. Hence Theorem~\ref{bwdbov} is never true for meromorphic functions.
\end{rem}

The converse of Theorem~\ref{bwdbov} is not true in general as given in the following example.
\begin{example}
	Let $f_{\lambda}(z)=e^z+z+ \lambda, \lambda \geq 0$.
	Let $\gamma$ be an unbounded  curve.
	If the set $A= \{\Re(z): z \in \gamma\}$ is bounded then  $\{\Im(z): z \in \gamma\}$ is unbounded.
	But, by the mapping property of $e^z$, there exists  $M>0$ such that $|e^z| \le M$ for all $z \in \gamma$.
	Since $|f_{\lambda}(z)| \geq ||z+ \lambda|-|e^z|| \geq  |z+ \lambda|-M  $ and
	$ |z+ \lambda|-M  \to \infty$ as $z \to \infty$, $f_{\lambda}(\gamma)$ is unbounded.
	If $A$ is unbounded then consider a sequence $z_n$ in $\gamma$ such that $\Re(z_n) \to +\infty$ or $-\infty$ as $n \to \infty$. If $\Re(z_n) \to +\infty$ then $e^{z_n} \to \infty$. Otherwise $e^{z_n} \to 0$.
	Thus, in each case, $f_{\lambda}(z_n) =e^{z_n} +z_n + \lambda \to \infty$ as $n \to \infty$.
	Therefore, $f_{\lambda}(\gamma)$ is unbounded. By Theorem~\ref{boviff}, $\infty$ is the $\textit{bov}$ of $f_{\lambda}$.
	
	\par
	Note that if $L_k =\{z: \Im(z)= \pi k\}, k \in \mathbb{Z}$ then $f_{\lambda} (z) \in L_k$ for all $z \in L_k$.
	The function $f_\lambda$ has no finite asymptotic value and the critical values are $\lambda -1+ i \pi (2k +1)$. Further, if $z^*$ is a fixed point of $f_{\lambda}$ then its multiplier is $1- \lambda$. In other words, all the fixed points of $f_{\lambda}$ are attracting when $0< \lambda < 2$. Note that $z_k=\log |-\lambda|+i \pi k$ is a fixed point of $f_{\lambda}$ for all odd $k$ and these are the only fixed points.
	We assert that the attracting basin of $z_k$ is unbounded for $0< \lambda <1$.
	   \begin{figure}[ht!]
		\centering
		\includegraphics[scale=0.5]{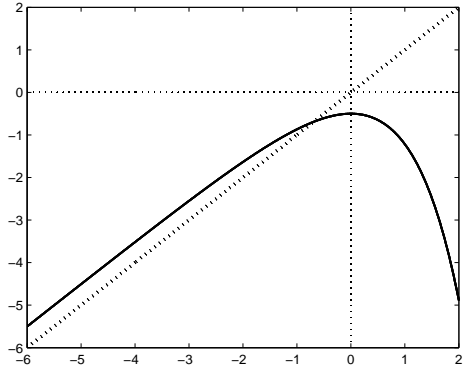}
		\caption{Graph of $\phi(x)=-e^x+x+ \lambda, 0< \lambda <1$. \label{bovbutnotbwd-example}}
	\end{figure}
	\par
	For odd $k$, $T_{i \pi k}^{-1}(f_\lambda( T_{i \pi k}(x)))= \phi(x)$ where $T_{i \pi k}(x)= x+ i \pi k$  and $\phi(x)= -e^x+x +\lambda$. The function $\phi(x)$ has a fixed point at $\log  \lambda $ and this is negative as $0 < \lambda <1 $. Further, $\phi(x)$ is increasing on the negative real axis, attains its maximum at $0$ and then decreases. It also follows that $\phi(x)> x$ for all $x < \log  \lambda $. This means that $\phi^n(x) \to \log \lambda $ as $n \to \infty$ for all $x< \log  \lambda $. Thus $f_{\lambda}^n(z) \to z_k ~\mbox{as}~ k \to \infty ~\mbox{for all}~ z \in L_k$ for odd $k$ and such that $\Re(z) < \log \lambda$. Therefore, the basin of attraction of $z_k$ is unbounded. Note that if a function has a Baker wandering domain then all its Fatou components including the Baker wandering domains are bounded. This gives that $f_{\lambda}$ has no Baker wandering domain.
\label{example3}
\end{example}

Sufficient conditions for existence of $\textit{bov}$ for meromorphic functions follows from some earlier work reported by Nayak ~\cite{tkk}.
\begin{thm} Let $f \in M$  have a single omitted value $a$,  $a\in\mathcal{J}(f)$ and $\mathcal{J}_{a}$, the component of the Julia set containing $a$ is singleton. If there is a non-contractible Jordan curve in the Fatou set and a subsequence $m_k$  such that $\lim_{k \to \infty} f^{m_k}(\gamma) =\infty$ and $f^{m_k}(\gamma)$ surrounds $a$ for all large $k$ then  $f$ has $\textit{bov}$.
	\label{suffbovmero}
\end{thm}
\begin{proof}
	Let $\gamma_{m_k}$ denote the image of $\gamma$ under $f^{m_k}$.
	Since $\lim_{k \to \infty}\gamma_{m_k} = \infty$ and  $\gamma_{m_k}$ surrounds $a$, the asymptotic path corresponding to $a$ ( which is an unbounded curve ) intersects $\gamma_{m_k}$, say at $z_k$ for infinitely many $k$.
Further, it follows that $\gamma_{m_k}$ surrounds every finite complex number  for all large $k$.
In particular, $\gamma_{m_k}$ surrounds a pole  for all large $k$ and this gives that $f(\gamma_{m_k})$ surrounds $a$ for all large $k$ (See Lemma~1\cite{tk-zheng}).
 Consequently, $f(\gamma_{m_k}) \to a~\mbox{or~} \infty$. It is obvious that $z_k \to \infty$ and $f(z_k) \to a$ as $k \to \infty$. Therefore,  $f(\gamma_{m_k}) \to a$ as $k \to \infty$.
	\par
	Let $B$ be a neighborhood of $a$ and the boundary of $f^{-1}(B)$ have an unbounded component $\beta$.
Then $f(\beta) \subseteq \partial B$. Note that $\beta$ intersects $\gamma_{m_k}$, say at $w_k$ for infinitely many $k$ and hence $f(w_k) \in \partial B$. But by the conclusion of the last paragraph, $f(w_k) \to a$ as $k \to \infty$. This is a contradiction. Therefore, each component of the boundary of $f^{-1}(B)$ is bounded. Hence $f$ has $\textit{bov}$.
\end{proof}
An immediate corollary follows. A Julia component is called buried if it is not contained in the boundary of any Fatou component.
\begin{cor} Let $f \in M$ have a single omitted value $a$ and $a\in\mathcal{J}(f)$.
 If $\mathcal{J}_{a}$, the component of the Julia set containing $a$ is singleton and non-buried  then $a$ is the $\textit{bov}$ of $f$.
 \label{singletonnonburied}
\end{cor}
\begin{proof}
	Since  $\mathcal{J}_{a}$ is singleton and non-buried, there is a Fatou component $U$ of $f$ such that $\mathcal{J}_{a} \subset \partial U$. It can be shown that $U$ is infinitely connected. Considering a non-contractible
	Jordan curve $\gamma$ in $U$, it is seen that $f^{m_k}(\gamma)$ surrounds $a$ and $f^{m_k} \to \infty$ or $a$ on $\gamma$ as $k \to \infty$ for some subsequence $m_k$ (See Lemma~1\cite{tk-zheng}).
	\par
	If $U$ is wandering then $f^{m_k} \to a$ on $\gamma$ is possible only when $\mathcal{J}_a$ is singleton.
	But in that case, $\mathcal{J}_a$ is buried which is contrary to the assumption.
	Therefore $f^{m_k} \to \infty$ on $\gamma $ whenever $U$ is wandering and we are done by Theorem~\ref{suffbovmero}.
	\par Let $U_m$ be periodic for some $m  \geq 0$. It can be shown that $U_k$ is infinitely connected for all $k \geq m$ (See Lemma 3.5~\cite{tkk}). Therefore $U_m$ is neither a Siegel disk nor a Herman ring. Since $a$ or $\infty$ is a limit function of the sequence of iterates of $f$ on $U_m$, the Fatou component $U_m$ can neither  be an attracting domain nor a parabolic domain. Therefore $U_m$ is a Baker domain and $f^{m_k} \to \infty~\mbox{or}~ a$ as $k \to \infty$ on $\gamma$. If  $f^{m_k} \to \infty$ as $k \to \infty$ on $\gamma$ then we are done by Theorem~\ref{suffbovmero}.
Otherwise, that means, if $f^{m_k} \to a$ as $k \to \infty$ on $\gamma$ then $f^n(a)=\infty$ for some $n$. Consequently, $f^{m_k +n}(\gamma)$ surrounds $a$ by Lemma~1\cite{tk-zheng} and $f^{m_k+n} \to \infty$ as $k \to \infty$ on $\gamma$.
 Now, the proof follows from Theorem~\ref{suffbovmero}.
\end{proof}

\section{Completely invariant Fatou components}
Following lemma deals with Fatou components when the $\textit{bov}$ is in the Fatou set.
\begin{lemm}
	Let $f \in M$ have $\textit{bov}$ $a$ and $a$ is in the Fatou set of $f$.
	Then the following are true about $f$.
	\begin{enumerate}
		\item There exists exactly one unbounded Fatou component.
		\item If $U$ is a completely invariant Fatou component then $a \in U$.
		\item If $a$ is in an unbounded Fatou component $V$ then $V$ is completely invariant.
	\end{enumerate}
	\label{bovinfatouset}
\end{lemm}
\begin{proof}
	\begin{enumerate}
		\item
		Let $U$ be the Fatou component containing $a$. Then there exists a ball $B_r (a)$ centered at $a$ contained in $ \mathcal{F}(f)$. By definition of $\textit{bov}$, $f^{-1}(B_r(a))$ is an infinitely connected unbounded  subset of the Fatou set. This means that there is a single Fatou component, say $V$ such that $f(V)=U$. Clearly $V$ is unbounded. If $V$ is the only Fatou component of $f$  we are done. Note that $U=V$ in this case. Otherwise, there is a  Fatou component $W$ different from $V$. The proof will be complete by showing that $W$ is bounded. The Fatou component $W_1$ containing $f(W)$ is clearly different from $U$. Now, it follows from Lemma~\ref{bovnotinball} that each component of $f^{-1}(W_1)$ is bounded. In particular, $W$ is bounded.
		
		\item If $U$ is a completely invariant Fatou component of $f$ then $U$ is unbounded and this is the only unbounded Fatou component of $f$  by (1) of this Lemma. It follows from the definition of $\textit{bov}$ that there exists $r>0$ such that $f^{-1}(B_r(a))$ is an unbounded subset of the Fatou set in such a way that it intersects $U$. Thus $f^{-1}(B_r(a)) \subset U$ and consequently $ B_r(a) \subset U$ since $U$ is completely invariant. Therefore $a\in U$.
		
		\item Let $a\in V$ and consider a ball $B_r(a)$ contained in $ V$. Then $f^{-1}(B_r(a))$ is a connected and unbounded subset of the Fatou set. Since $V$ is the only unbounded Fatou component, it contains $f^{-1}(B_r(a))$ giving that  $f^{-1}(V) \subseteq V$.
		We know that $f$ is continuous on $V$ which gives that $f(V)$ is connected. Since $V$ is backward invariant, we can choose a $z\in V$ such that $f(z)\in V$. Then $f(V)\cap V\neq \emptyset$. This implies $V$ is the Fatou component containing $f(V)$. Therefore $V$ is forward invariant.	Hence V is completely invariant.
		
	\end{enumerate}
\end{proof}

It is a conjecture that the number of completely invariant Fatou components of a meromorphic function is at most two. This has already been proved for rational functions, transcendental entire functions and transcendental meromorphic functions of finite type. Theorem $7$~\cite{tk-zheng} states that if $f$ is a transcendental meromorphic function with a single omitted value and $f$ has at least one critical value then the number of completely invariant Fatou components is at most two. The following result confirms a stronger version of the conjecture if the omitted values is a $\textit{bov}$. In the following theorem, let $CV_f$ and $O_f$ denote  the sets of all critical values and omitted values of $f$ respectively.

\begin{thm}
	If $f \in M$ has $\textit{bov}$ then the number of completely invariant Fatou components is at most one.
	\label{onecifc}
\end{thm}
\begin{proof}
	Let $a$ be the $\textit{bov}$ of $f$.
	If there exists two completely invariant Fatou components of $f$, say $U_{1}$ and $U_{2}$ then each is simply connected \cite{bkr-kts-yin} and unbounded.
	
	Suppose $f$ has a critical value. Then either $CV_{f}\subset U_{1}$ or $O_{f}\subset U_{1}$ by Lemma 5  of ~\cite{tk-zheng}. Without loss of generality we can assume  that $O_{f}\subset U_{1}$. So the $\textit{bov}$ $a$ is in $U_{1}$. By Lemma~\ref{bovinfatouset}(1), there exists a unique unbounded Fatou component. This is a contradiction as there are two unbounded Fatou components.
	\par
	Suppose $f$ has no critical value. Since $a$ is the $\textit{bov}$, it is the
	only asymptotic value of $f$. Thus $f$ is of finite type and it has no Baker domains or wandering domains.
Since $U_1$ and $\ U_{2}$ are completely invariant, none of them is either a Herman ring or a Siegel disk.
	Thus $U_i, i=1,2 $ is an attracting or a parabolic domain and contains at least one singular value. But $a$ is the only singular value of $f$. This is a contradiction and it is proved that the number of completely invariant Fatou components is at most one.
\end{proof}

From some earlier work of Nayak and Zheng, a sufficient condition for non existence of completely invariant Fatou components follows.
\begin{rem}
	Let $f \in M$  have  $\textit{bov}$ $a$. If $a\in \mathcal{J}(f)$ and $ \mathcal{J}_{a} $ is bounded then $f$ has no completely invariant Fatou component.
\end{rem}
\begin{proof}
	If  the Julia component $\mathcal{J}_{a} $ containing $a$ is singleton then the claim follows from Corollary 1 (i) of  ~\cite{tk-zheng}. If  $\mathcal{J}_{a} $ is not singleton then there is a multiply connected Fatou component which must be a Herman ring by Theorem 1 ~\cite{tk-zheng}. But a Herman ring implies non existence of completely invariant Fatou component.
\end{proof}

\section{Questions further}
We put forth some questions arising out of this article.\par
Theorem~\ref{bwdbov} proves that if an entire function has a Baker wandering domain then it has $\textit{bov}$. That the converse is not true in general was exhibited by Example~\ref{example3}. In this context the following makes sense.
\begin{que}  Characterize all entire functions for which existence of $\textit{bov}$  implies the existence of Baker wandering domain.
\end{que}
In view of Corollary~\ref{singletonnonburied}, the following question remains unanswered.
\begin{que} Let $f \in M$ have a single omitted value $a$,
$\mathcal{J}_a$ is singleton and buried. Is it true that $a$ is $\textit{bov}$ of $f$?
\end{que}
If an omitted value of a function is a pole then it has no Herman ring. Also, if all the poles of a function having an omitted value are multiple, then it has no Herman ring. If a meromorphic function has an omitted value, then it has no Herman ring of period one and two. All these are proved in \cite{tk} and are true for functions with $\textit{bov}$. We incline to make the following conjecture.
\begin{que} Meromorphic functions with $\textit{bov}$ have no Herman ring.
\end{que}
It follows from Theorem~\ref{oneav} that  a function with $\textit{bov}$ has no asymptotic value other than the $\textit{bov}$. This restriction seems to simplify the investigation of the dynamics of the function. Thus the following is worth-doing.
\begin{que} Investigate the dynamics of entire  and meromorphic functions with $\textit{bov}$.
\end{que}

\par

 A comparison is made between entire and meromorphic functions with $\textit{bov}$ in the following table.
 \newpage

\begin{table} \caption{Comparison}
	Comparison of properties of entire and meromorphic functions with $\textit{bov}$.
	\centering 
	\begin{tabular}{  @{}ll} 
		\hline 
		$f\in E$ &    $f\in M$  \\ [0.5ex] 
		
		\hline 
		1. The $\textit{bov}$ is  always   $\infty$. &    1. The $\textit{bov}~ a$ is always finite. \\
	
		2. The $\mathcal{J}_\infty$    is singleton  if and only if   &  2. It has no Baker  wandering  \\
		   a  Baker  wandering  domain exists. &    domain.  \\
		
		3. The $\textit{bov}$ 	is always in the Julia set. & 3. The $\textit{bov}$ can be in the Fatou set.\\

		4. If $\mathcal{J}_\infty$   is not singleton then  &  4. If $\mathcal{J}_\infty$ is not singleton then every\\
	
		all the Fatou components  & Fatou component is simply connected \\
		are simply connected. &or lands on a Herman ring. \\
		5. If  $\mathcal{J}_\infty$ is singleton  &  5. If  $\mathcal{J}_a$  is singleton\\
		
	  then it is buried. &   then nothing is known. \\ \\
	\end{tabular}
	\label{table:nonlin} 
\end{table}


\end{document}